\documentclass[12pt]{article}
\usepackage{e-jc_for_arxiv}

\usepackage{amsmath}
\usepackage[final]{graphicx}
\usepackage{color}
\usepackage{microtype}

\dateline{Jun 19, 2020}{Jul 21, 2021}{TBD}
\MSC{05C15, 05D40}

\usepackage{algorithm}
\usepackage[noend]{algpseudocode}

\algnewcommand\algorithmicforeach{\textbf{for each}}
\algdef{S}[FOR]{ForEach}[1]{\algorithmicforeach\ #1\ \algorithmicdo}

\DeclareMathOperator\dist{dist}
\DeclareMathOperator\FIX{\mathrm{\textbf{FIX}}}

\let\le\leqslant
\let\ge\geqslant
\let\leq\leqslant

\let\epsilon\varepsilon

\newcommand{\blank}{\mathsf{Blank}}

\title{Tight Bounds on the Clique Chromatic Number}

\author{Gwena\"el Joret\thanks{Supported by an ARC grant from the Wallonia-Brussels Federation of Belgium and a CDR grant from the National Fund for Scientific Research (FNRS).}\\
\small Computer Science Department\\[-0.8ex]
\small Universit\'e libre de Bruxelles\\[-0.8ex] 
\small Brussels, Belgium\\
\small\tt gjoret@ulb.ac.be\\
\and
Piotr Micek\thanks{Supported by the National Science Center of Poland under grant no.\ 2018/31/G/ST1/03718.}\\
\small Theoretical Computer Science Department \\[-0.8ex]
\small Jagiellonian University \\[-0.8ex] 
\small Krak\'ow, Poland\\
\small\tt piotr.micek@uj.edu.pl\\
\and 
Bruce Reed \\
\small School of Computer Science\\[-0.8ex]
\small McGill University\\[-0.8ex] 
\small Montreal, Canada\\
\small\tt breed@cs.mcgill.ca\\
\and 
Michiel Smid\thanks{Supported by NSERC.}\\
\small School of Computer Science\\[-0.8ex]
\small Carleton University\\[-0.8ex] 
\small Ottawa, Canada\\
\small\tt michiel@scs.carleton.ca\\
}

\begin{document}

\maketitle

\begin{abstract}
The clique chromatic number of a graph is the minimum number of colours needed to colour its vertices so that no inclusion-wise maximal clique, which is not an isolated vertex, is monochromatic. 
We show that every graph of maximum degree $\Delta$ has clique chromatic number $O\left(\frac{\Delta}{\log~\Delta}\right)$.  We  obtain as a corollary that every $n$-vertex graph has clique chromatic number $O\left(\sqrt{\frac{n}{\log ~n}}\right)$. 
Both these results are tight.
\end{abstract}

\section{Overview and History}

A {\em clique colouring} of a graph $G$ is a colouring of its vertices so that no inclusion-wise maximal clique of size at least two is monochromatic. 
The {\em clique chromatic number} of $G$, introduced in \cite{BGGPS}, is the minimum number of colours in a clique colouring of $G$.  
We show that every graph of maximum degree $\Delta$ has clique chromatic number $O\left(\frac{\Delta}{\log~\Delta}\right)$.  We  obtain as a corollary that every $n$-vertex graph has clique chromatic number $O\left(\sqrt{\frac{n}{\log ~n}}\right)$. 
Both these results are tight. 
Indeed, if the graph is triangle free, then the clique chromatic number coincides with the usual chromatic number, and these two bounds are best possible for the chromatic number~\cite{K}. 

The first results on the clique chromatic number of which we are aware were obtained in 1991 and concerned comparability graphs of partial orders.
A {\em partial order} $(X,\le)$ consists of a ground set $X$ and a reflexive, transitive, antisymmetric relation $\le$. 
Two elements $x$ and $y$ of $X$ are {\it comparable}  if either $x \le y$ or $y \le x$. 
The {\em comparability graph} of this partial order has vertex set $X$, and two vertices are joined by an edge if they are comparable.

It is easy to see that  comparability graphs have clique chromatic number at most two.
In \cite{DSSW},  Duffus, Sands, Sauer, and Woodrow, showed that there are  complements of comparability graphs whose clique chromatic number exceeds two, and asked if the clique chromatic number of perfect graphs, a superclass of  both   the comparability graphs and their complements was bounded by a constant. 
In \cite{DKT}, Duffus, Kierstead, and Trotter show that  the complements of comparability graphs have clique chromatic number at most $3$. 
The question about the clique chromatic number of perfect graphs was popularized by Jensen and Toft in their  1994 manuscript~\cite{JT} setting out the most interesting questions about colouring graphs. It was recently answered in the negative by Charbit, Penev, Thomass\'e, and Trotignon~\cite{CPTT}. 

There has been significant interest on the clique chromatic number of  other special classes of graphs (see e.g.\ \cite{BGGPS,MS}), and  the random graph $G_{n,p}$ for 
various values of $p$~\cite{MMP}. As far as we are aware, the only results concerning arbitrary graphs are unpublished results of Kotlov which are mentioned in~\cite{BGGPS}. That paper 
sets out that  Kotlov showed that the clique chromatic number of an $n$-vertex graph is at most $2\sqrt{n}$ and notes that it is unknown whether the ratio of the clique chromatic number to 
$\sqrt{n}$ goes to zero as $n$ goes to infinity.  Our main result shows that, in fact, it does. 

\section{Our Results}

In what follows $N_G(v)=N(v)$, the neighbourhood of $v$ in $G$, is the set of vertices of $G$ joined to $v$ by an edge.  
Our main result is the following:

\begin{theorem}
\label{maintheorem}
For every  $\epsilon>0$, there exists a $\Delta_\epsilon$ such that, every graph $G$ with maximum degree $\Delta \ge \Delta_\epsilon$ 
has  clique chromatic number  at most $\frac{(1+\epsilon)\Delta}{\log~\Delta}$. 
\end{theorem}

From which we obtain:

\begin{corollary}
The clique chromatic number of an $n$-vertex graph $G$  is $O\left(\sqrt{\frac{n}{\log~n}}\right)$. 
\end{corollary}

\begin{proof}[Proof of corollary]
We choose a maximal sequence $v_1, \dots,v_k$ of  vertices of $G$ such that setting $G_0=G$, and $G_i=G_{i-1} -v_i-N_{G_{i-1}}(v_i)$, 
we have: $|N_{G_{i-1}}(v_i)| \ge  \sqrt{n \log ~n}$.
Clearly $k \le  \sqrt{\frac{n}{\log ~n}}$.
 Now, $G_k$ has maximum degree at most  $ \sqrt{n\log ~n}$.
Hence by Theorem \ref{maintheorem}, it has  a clique colouring with  $O\left(\sqrt{\frac{n}{\log ~n}}\right)$ colours.  We extend such a clique  colouring of 
 $G_k$, which uses colours not including $\{0,1,...,k\}$,  to a clique colouring of $G$ by colouring, for $1 \le i \le k$, 
each $v_i$ with colour $0$, and $N_{G_{i-1}}(v)$ with colour $i$.

Since the colours we use on $V(G_k)$ are  distinct from the colours we use on  $V(G)-V(G_k)$, and we construct a clique colouring in $G_k$, if there is a monochromatic maximal clique with at least two vertices in $G$, it must be coloured $i$ for some $i \le k$. Now, the vertices coloured 0 are a stable set, and the vertices coloured $i$  for $1 \le i \le k$ are all adjacent to $v_i$ which is not coloured $i$. So, we do indeed have a clique colouring of $G$. 
\end{proof}

\begin{proof}[Proof of main theorem]
In  2018, Molloy~\cite{M} proved 

\begin{theorem}[\cite{M}, Theorem 1] 
For every $\epsilon>0$,  there exists a $\Delta_\epsilon$ such that, every triangle-free graph $G$ with maximum degree $\Delta \ge \Delta_\epsilon$ 
has list chromatic number  at most  $\frac{(1+\epsilon)\Delta}{\log~\Delta}$. 
\end{theorem} 

To obtain the statement of Theorem 1 from the statement of this theorem we simply need to remove {\it triangle-free} and change {\em  list} to  {\it clique}. 
To obtain the proof of Theorem 1 from Molloy's proof we make only a few  minor adjustments. There is one small subtle difference which we need to point out
however. 

The number of colours we can use 
is $q=\lfloor \frac{(1+\epsilon)\Delta}{\log~\Delta} \rfloor$, and our vertices are each assigned the list ${\cal C}_v =\{1,2,\dots,q\}$ of colours. (Since Molloy is dealing with list colouring
the list of colours assigned to his vertices are preasssigned. He also neglects to round $q$ down in his definition but he is indeed using the same value of $q$ as we do throughout).  

By a  {\em  partial clique colouring}, we mean a colouring of some of the vertices of $G$ using colours in $\{1,\dots,q\}$ such that no maximal clique of size at least two is monochromatic. 
We assign the uncoloured vertices a new colour $\blank$, and thus our condition is actually that the vertices of any monochromatic maximal clique of size at least two are coloured $\blank$. The main work in the proof is to show that there is a partial colouring with the following property: \\

\begin{minipage}{0.02\textwidth}
\centering
\vspace{-1cm}
\begin{equation} 
\label{prop}
\end{equation}
\end{minipage}
\quad 
\begin{minipage}{0.82\textwidth}
The partial clique  colouring can be extended  by assigning a colour  from  ${\cal C}_v=\{1,\dots,q\}$ to each uncoloured vertex $v$ such that none of these vertices is in a monochromatic edge.
\end{minipage}

$ $

If we are given a partial clique  colouring and an extension of it as in \eqref{prop}, then the resulting colouring is a clique colouring of $G$. So, to show the theorem, it suffices to prove~\eqref{prop}. 
 
In proving that  a  partial  clique colouring  satisfies \eqref{prop}, we will need to consider for each vertex $v$, the set of colours available to be assigned to $v$,
and the set of  uncoloured neighbours for $v$. 
 For a partial  clique colouring $\sigma$, we define  the following  parameters for each vertex $v$:
\begin{itemize}
\item $L_v$, the   union of $\{\blank\}$ and ${\cal C}_v-\{\sigma(w) | ~w \in N(v)\}$;
\item for every colour  $c$,  the following set $T_{v,c}$ 
\[
T_{v,c} = \left\{
\begin{array}{ll}
\{w\; |~w  \in N(v), \; \sigma(w)=\blank, \; c \in L_w\} & \textrm{ if } c\neq\blank \\
\emptyset & \textrm{ otherwise.}
\end{array}
\right.
\]
\end{itemize}

We  set $L=\Delta^{\epsilon/2}$ and define the following two  {\it flaws} (events) for a partial clique colouring:
\begin{itemize}
\item $B_v: |L_v|<L$,
\item $Z_v:  \sum_{c \in L_v} |T_{v,c}| >\frac{1}{10}L \cdot |L_v|$.
\end{itemize}

If a flaw $f$ is $B_v$ or $Z_v$  then we say $v$ is the vertex of $f$ and define $v(f)$ to be $v$. 
The proof of Lemma 5 of \cite{M} shows that if a partial colouring has no flaws, then it satisfies property \eqref{prop}. 
Specifically, it shows  that if every uncoloured vertex is assigned a uniformly random colour from $L_v - \{\blank\}$,  
then with positive probability we obtain an extension in which none of these vertices is in a monochromatic edge. 
The proof is a standard application of the local lemma. 
Thus we have the following lemma. 

\begin{lemma}[corresponds to Lemma 5 in \cite{M}, proof is verbatim the same]
\label{lem5}
Suppose we have a partial clique colouring $\sigma$ with no flaws, i.e., such that neither $B_v$ nor $Z_v$ holds for any vertex $v\in V(G)$. 
Then we can extend $\sigma$ by colouring each blank vertex $v$ with a colour in $\mathcal{C}_v$ to obtain a clique colouring of $G$. 
\end{lemma}

In light of this lemma, we need only show that there is a partial colouring with no flaws. To do so, we adopt Molloy's procedure, which starts with the colouring in which every vertex is assigned  $\blank$, and then in each iteration  uncolours and then randomly recolours the  vertices in the neighbourhood of $v=v(f)$  for some carefully chosen flaw $f$ until no such flaws exist. 
The latter is done using a recursive algorithm called $\FIX(f,\sigma)$ in~\cite{M}. 
In this algorithm, when recolouring $N(v)$, each vertex $u$ is assigned a colour from $L_u$ uniformly at random.  
For our purposes we need to make a subtle change in the recolouring procedure. 
Define
\begin{itemize}
	\item $L^v_u$ as the set of colours in $\mathcal{C}_u$ not appearing in $N(u) - N(v)$, along with $\blank$.  
\end{itemize}
(That is, when computing $L^v_u$ we ignore the neighbours of $u$ that are in $N(v)$.) 
In our version of the algorithm $\FIX(f,\sigma)$, when recolouring $N(v)$, each vertex $u$ is assigned a colour {\em from $L^v_u$} uniformly at random, instead of $L_u$.   
The modified algorithm is described in Algorithm~\ref{algo-fix}. 

\begin{algorithm}
\caption{$\FIX(f,\sigma)$}
\label{algo-fix}
\begin{algorithmic}[1]
\State $v \gets$ $v(f)$
\ForEach{$u\in N(v)$}\label{step_mainLoop}
\State $\sigma(u) \gets$ a colour from $L^v_{u}$ chosen uniformly at random
\EndFor
\While{there are any flaws $B_w$ with $\dist(w,v)\leq 2$ or $Z_w$ with $\dist(w,v)\leq 3$}\label{step_mainLoop2}
\State $g \gets$ first such flaw
\State $\sigma \gets \FIX(g,\sigma)$
\EndWhile
\State \textbf{return} $\sigma$
\end{algorithmic}
\end{algorithm}

Let us remark that if the graph is triangle free, then $L^v_{u} = L_u$ always holds; thus, in the triangle-free case the two versions of the algorithm $\FIX$ are the same. 

Note that, by the definition of $L^v_u$, if the recolouring creates a monochromatic clique  $C$  of size at least two which is not coloured $\blank$, then all of the vertices of the clique are in $N(v)$. 
But then $C$ is not maximal as we can add $v$ to it. 
So recolouring leaves us with a partial clique colouring, as desired.  

Note also that if a call to $\FIX(f,\sigma)$ terminates, then we made some progress towards correcting the flaws. 

\begin{observation}[corresponds to Observation 6 in \cite{M}, proof is verbatim the same]
In the partial clique colouring returned by $\FIX(f,\sigma)$:
\begin{itemize}
    \item the flaw $f$ does not hold, and
    \item there are no flaws that did not hold in $\sigma$. 
\end{itemize}
\end{observation}

Thanks to the above observation, to obtain a partial clique colouring without any flaw it suffices to start with the all blank colouring, and then call $\FIX$ at most once for each of the at most $2n$ flaws of that colouring, where $n=|V(G)|$. 
Thus, to complete the proof of our theorem, it suffices to show that, with positive probability, these at most $2n$ calls to $\FIX$ all terminate. 
Molloy shows precisely this in the triangle-free case, the proof consists of Lemmas 7, 8, and 9 in \cite{M}. 
As it turns out, replacing the lists $L_u$ with the lists $L^v_u$ in the algorithm $\FIX$ is the only modification that needs to be done in Molloy's approach to prove our theorem. 
These three lemmas and their proofs can then be taken verbatim from \cite{M} up to straightforward changes reflecting the modified lists used in the recolouring step. 
For completeness, we now consider each lemma in turn and summarize the changes that need to be done in their proofs. 

We begin with the key lemma, bounding the probability that there is a flaw at $v$ after recolouring $N(v)$.  
The setup for this lemma is as follows. 
Each vertex $u\in N(v)$ has some some current list $L_u^v$ containing $\blank$ and perhaps other colours. 
We give each vertex $u\in N(v)$ a colour chosen randomly from $L_u^v$, where the choices are made uniformly and independently. 
This determines $L_v$ and $T_{v, c}$. 

\begin{lemma}[Lemma 7  in \cite{M}]
\label{lem7} 
$ $
\begin{enumerate}
    \item[(a)] $\mathbf{Pr}(|L_v| < L) < \Delta^{-4}$.
    \item[(b)] $\mathbf{Pr}(\sum_{c\in L_v}|T_{v,c}| > \frac{1}{10} L \times |L_v|) < \Delta^{-4}$.
\end{enumerate}
\end{lemma}

The proof of this lemma is obtained by taking that of Lemma 7 in \cite{M} and replacing every occurrence of the list $L_u$ (with $u$ a neighbour of $v$) with the list $L_u^v$, to match our modified recolouring procedure. 
There are 6 such occurrences before part (a) of the proof, 7 in part (a) (including one in the proof of the claim), and 6 in part (b).  
There is also another small change to make in part (b) of the proof, namely, when considering colors $c\notin \Psi$, the line 
    \[
    \mathbf{E}(|T_{v,c}|) = \sum_{u: c\in L_u} \frac{1}{|L_u|} < \rho(c) \leq \Delta^{\epsilon/4}
    \]
    should become 
    \[
    \mathbf{E}(|T_{v,c}|) \leq \sum_{u: c\in L^v_u} \frac{1}{|L^v_u|} < \rho(c) \leq \Delta^{\epsilon/4}.
    \]
The reason the equality sign becomes an inequality is the following: For a vertex $u\in N(v)$ with $c\in L^v_u$, the probability that $u$ gets coloured $\blank$ in our recolouring is exactly $\frac{1}{|L^v_u|}$. However, not all these vertices $u$ will contribute to $T_{v,c}$, only those such that the colour $c$ is in $L_u$ after recolouring\footnote{Note that, while $L^v_u$ is unaffected by our recolouring, $L_u$ on the other hand is, which is why we need to distinguish between $L_u$ before and after recolouring.}, hence the inequality.  

Next we turn to Lemmas 8 and 9 from \cite{M}. 
These lemmas are used to perform a Moser-Tard\'os type analysis of the algorithm $\FIX$. 
First we introduce some notations, adapted from~\cite{M}. 
Suppose that we are given a collection of lists $\mathcal{L}=\{L_u^v: u \in N(v)\}$ of available colours for the neighbours of a vertex $v$. 
Let $\mathcal{B}(\mathcal{L})$ ($\mathcal{Z}(\mathcal{L})$) be the set of all colour assignments from these lists such that the flaw $B_v$ (respectively $Z_v$) holds. 
Then by Lemma~\ref{lem7} we have $|\mathcal{B}(\mathcal{L})|, |\mathcal{Z}(\mathcal{L})| < \Delta^{-4} \prod_{u\in N(v)} |L_u^v|$, which is a key fact used in the proof of the last lemma below. 

The setup for Lemma 8 from \cite{M} is as follows. 
Let $n:=|V(G)|$ and consider any of the at most $2n$ `root' calls to $\FIX$ we make to fix our initial partial clique colouring, say we call it with flaw $f$ and partial clique colouring $\sigma_0$. 
(Thus, for the very first call $\sigma_0$ is the all blank colouring.) 
Considering the subsequent recursive calls made by the algorithm, we let $\sigma_t$ denote the partial clique colouring just after the $t$-th recolouring step (the for loop in lines 2 and 3 of algorithm $\FIX$). 
We let $R_t$ describe the random choices that have been made during the $t$-th recolouring step, and finally we let $H_t$ be a ``log'' listing in a concise manner the recursive calls that were made during this particular call of $\FIX$ in the while loop (lines 4--6 of the algorithm), see~\cite{M} for the precise definition of the log.   

\begin{lemma}[Lemma 8 in \cite{M}]
Given $\sigma_0, \sigma_t, f, H_t$ we can reconstruct the first $t$ steps of $\FIX$. 
\end{lemma}

Thus, knowing $\sigma_0, \sigma_t, f, H_t$ is enough to deduce all the random choices that were made up to step $t$, that is, $R_1, R_2, \dots, R_t$. 
The proof of this lemma is obtained from that of Lemma 8 in \cite{M} by replacing every occurrence of the list $L_u$ (with $u$ a neighbour of $v$) with the list $L_u^v$. 
Also, one observation needs to be slightly adapted: 
In 10th and 11th line in the proof of Lemma 8 in \cite{M}, Molloy remarks crucially that in his case the set $\mathcal{L}=\{L_u: u \in N(v)\}$ does not change after recolouring $N(v)$ because the graph is triangle free.  
In our setting, where the graph is arbitrary and $\mathcal{L}=\{L_u^v: u \in N(v)\}$, this remains true by the very definition of $L_u^v$. 

Finally, using the previous two lemmas, Lemma 9 from \cite{M} shows that any call $\FIX(f, \sigma)$ terminates with positive probability:

\begin{lemma}[Lemma 9 in \cite{M}]
For any partial clique colouring $\sigma$ and any flaw $f$ of $\sigma$, the probability that $\FIX(f, \sigma)$ performs at least $2n$ recolouring steps is at most 
$\Delta^{-n/2}$, where $n=|V(G)|$. 
\end{lemma}

The proof is verbatim the same as that of Lemma 9 in \cite{M}, with the list $L_u$ replaced with $L_u^{v(f_i)}$ in the definition of $\Lambda_i$.  

Given this lemma, the proof of our theorem follows: Starting with our initial all blank partial clique colouring, with probability at least $1 - 2n\Delta^{-n/2} > 0$ 
we perform at most $2n$ separate `root' calls $\FIX(f, \sigma)$ and eventually obtain a partial clique colouring without any flaws. 
Lemma~\ref{lem5} then implies that this partial clique colouring can be turned into a clique colouring. 

In summary, by slightly modifying the recolouring procedure in the algorithm $\FIX$ as we explained, and performing the straightforward changes described above, Molloy's proof becomes a proof of our Theorem. 
\end{proof}

\section{A Connection to Perfect Graphs and An Intriguing Conjecture}

We can associate to any graph $G$ the clique hypergraph ${\cal H}(G)$, whose vertices are $V(G)$ and whose edge sets are the maximal 
cliques of $G$ of size at least two.  Then the usual colouring of $G$  is a colouring of $V(G)$ so that no edge of ${\cal H}(G)$  contains two vertices of the same colour, while a clique colouring of $G$ is one in which no edge of ${\cal H}(G)$ is monochromatic. 

It seems possible that graphs $G$ that can be properly coloured with few colours would also permit colourings with  relatively  few colours such that no edge of ${\cal H}(G)$ is monochromatic. Of course, a colouring of the first type requires 
$\omega(G)$ colours, where $\omega(G)$ is the size of the largest clique in $G$. So, an obvious conjecture in this direction is that if $G$ can be coloured using $\omega(G)$ colours then $G$ has bounded clique chromatic number. 

However, insisting  that $G$ has chromatic number $\omega(G)$ does not tell us anything about the structure of $G$. Indeed, for any graph $H$, the disjoint union of $H$ and a clique with $|V(H)|$ vertices yields such a graph $G$. In particular there are such graphs $G$ with arbitrarily high clique chromatic number. 
A graph is {\it perfect}~\cite{B} if for every induced subgraph $H$ of $G$, the chromatic number of $H$ is $\omega(H)$. A second obvious, but again false, conjecture,
discussed above, is that perfect graphs have bounded clique colouring number. 

Perfection is however related to two colourings of the hypergraph of the {\em maximum} cliques of a graph. 
Following Hoang and McDiarmid~\cite{HM}, we say that a graph is {\em $k$-divisible} if it can be $k$-coloured so that no 
maximum  clique of size at least two is monochromatic.  Now, if $G$ has an $\omega(G)$ proper colouring then giving 
one  of the stable sets of the colouring colour one, and the others colour 2 shows that $G$ is $2$-divisible. 
Thus every induced subgraph of a perfect graph is $2$-divisible.

Now an odd cycle of size at least five is not $2$-divisible, because it is not bipartite. Thus if every induced subgraph of $G$ is $2$-divisible then $G$ contains no induced odd cycle of length at least five. Hoang and McDiarmid~\cite{HM} conjectured that this necessary condition is also sufficient. Combined with the Perfect Graph Theorem  this would yield: ``Every induced subgraph of a graph $G$ is $\omega(G)$ colourable if and only if every induced subgraph of both $G$ and its complement is $2$-divisible.''

\section*{Acknowledgement}

This research was carried out at the Eighth Bellairs Workshop on Geometry and Graphs. The authors would like to thank the organizers for the invitation to participate in the workshop, and McGill University for facilitating their work by making the venue available for the workshop. Montreal, Ottawa, Brussels, and Krak\'ow are all very cold in January. 
The authors would also like to thank Michael  S.O'B.\ Molloy for most of the proof, and the referees for their comments.


\begin{thebibliography}{DSS}

\bibitem{BGGPS}
\newblock G. Basco, S. Gravier, A. Gy\'arf\'as, M. Preissmann, and A. Seb\H{o}.
\newblock Colouring the maximal cliques of a Graph. 
\newblock {\em SIAM J.\ Discrete Mathematics}, 17(3):361--376, 2004.

\bibitem{B}
\newblock C. Berge.
\newblock  Perfect graphs.
\newblock Six papers on graph theory.  
\newblock Indian Statistical Institute Calcutta, 1--21, 1993.

\bibitem{CPTT}
\newblock  P. Charbit, I. Penev, S. Thomass\'e, and N. Trotignon,
\newblock Perfect graphs of arbitrarily large clique chromatic number. 
\newblock {\em J.\  Combinatorial Theory Ser.\  B}, 116:456--464, 2016. 

\bibitem{DSSW}
\newblock D. Duffus, B. Sauer, N. Sands, and R. E. Woodrow.
\newblock Two colouring all two element maximal antichains. 
\newblock {\em J.\  Combinatorial Theory Ser.\  A}, 57(1):109--116, 1991.

\bibitem{DKT}
\newblock D. Duffus, H. Kierstead, and W. T. Trotter.
\newblock Fibres and ordered set colouring. 
\newblock {\em J.\  Combinatorial Theory Ser.\  A}, 58(1):158--164, 1991.


\bibitem{HM}
\newblock C. Hoang and C. McDiarmid.
\newblock On the divisibility of graphs. 
\newblock  {\em Discrete Mathematics}, 242:145--156, 2002. 

\bibitem{JT}
\newblock T. Jensen and B. Toft.
\newblock Graph colouring problems. 
\newblock  John Wiley and Sons, 1994. 

\bibitem{K}
\newblock J.-H. Kim.
\newblock The {R}amsey number $R(3,t)$ has order of magnitude  $t^2/\log~t$.
\newblock  {\em Random Structures and Algorithms}, 7:173--207, 1995. 

\bibitem{MMP}
\newblock  C. McDiarmid, D. Mietsche, and P. Pra\l{}at.
\newblock  Clique colouring of binomial random graphs.
\newblock  {\em Random Structures and Algorithms}, 54(4):589--614, 2019. 

\bibitem{MS}
\newblock  B. Mohar and R. Skrekovski.
\newblock  The {G}rotzsch {T}heorem for the hypergraph of maximal cliques.
\newblock  {\em Electronic J.\  of Combinatorics}, 6, 13 pp, 1994. 

\bibitem{M}
\newblock M. Molloy.
\newblock The list chromatic number of graphs  with small clique number.
\newblock {\em J.\  Combinatorial Theory Ser.\  B}, 134:234--264, 2019. 

\end{thebibliography}
\end{document}